\documentclass{amsart}
\usepackage[utf8]{inputenc}
\usepackage{xcolor}
\usepackage{tablefootnote}
\usepackage[colorlinks=true,linkcolor=blue,urlcolor=blue,citecolor=blue, pagebackref]{hyperref}
\usepackage[alphabetic]{amsrefs}
\usepackage{enumitem,amssymb}
\usepackage{tikz-cd}

\newcommand\op{\operatorname}
\newcommand\C{\mathbb{C}}
\renewcommand{\op}{\operatorname}
\newcommand{\gO}{\mathfrak{g}(\mathcal{O})}

\newcommand{\cyc}{\op{Gr}_{G,c(\lambda, \mu, \nu)}}

\newtheorem{lemma}{Lemma}
\newtheorem{remark}{Remark}
\newtheorem{theorem}{Theorem}
\newtheorem{prop}{Proposition}
\newtheorem{corollary}{Corollary}
\newtheorem{example}{Example}

\newtheorem{notation}{Notation}
\newtheorem{summary}{Summary}
\newtheorem{listy}{List}

\author{}
\date{}

\begin{document}
\title[Root Components and Geometric Satake]{Multiplicity in Root Components Via Geometric Satake}
\maketitle

\begin{minipage}{\textwidth}
\begin{center}
{\small
\begin{tabular}{cc}MARC BESSON & $\text{SAM JERALDS}^{*}$\footnote[0]{$~^*$corresponding author} \\
University of North Carolina & University of North Carolina \\
Chapel Hill, NC 27599 & Chapel Hill, NC 27599 \\
{\tt marmarc@live.unc.edu }& {\tt sjj280@live.unc.edu }\\
\end{tabular}\\\vspace{0.3in}
\begin{tabular}{c}
JOSHUA KIERS\\
Ohio State University \\
281 W Lane Ave \\
Columbus, OH 43201 \\
{\tt kiers.2@osu.edu}
\end{tabular}
}
\end{center}
\begin{abstract} In this note we explicitly construct 
top-dimensional components of certain cyclic convolution varieties. These components correspond (via the geometric Satake equivalence) to irreducible summands $V(\lambda+\mu-N\beta) \subset V(\lambda) \otimes V(\mu)$ for $SL_{n+1}(\C)$, where $N\ge 1$ and $\beta$ is a positive root. 
Furthermore, we deduce from these constructions a nontrivial lower bound on the multiplicity of these subrepresentations when $\beta$ is not a simple root. 
Finally, we demonstrate that not all such top-dimensional components can be realized as closures of orbits. 
\end{abstract}

\\
\end{minipage}

\newpage

\section{Introduction}

Let $G$ be a reductive linear algebraic group over $\mathbb{C}$, and let $G^{\vee}$ denote its Langlands dual group, also a reductive linear algebraic group over $\mathbb{C}$. Then the celebrated geometric Satake equivalence, in broad strokes, asserts that the geometry of the affine Grassmannian of $G$, denoted $\op{Gr}_G$, is intimately connected to the combinatorics of the category of finite-dimensional representations, $\op{Rep}(G^{\vee})$, for the Langlands dual group. We approach a small part of this large field in this work. On the representation theory side, we will be interested in the tensor decomposition problem. Namely, let $V(\lambda)$ and $ V(\mu)$ be finite-dimensional irreducible representations of $G^{\vee}$. Since  every finite-dimensional $G^\vee$ representation is completely reducible, we have the following decomposition, writing a tensor product of representations as a direct sum of irreducible representations: $V(\lambda) \otimes V(\mu) = \bigoplus V(\nu_i)^{\oplus m^{\nu_i}_{\lambda, \mu}}$. The tensor decomposition problem seeks to understand which highest weights $\nu_i$ appear in this decomposition, as well as their multiplicities $m^{\nu_i}_{\lambda, \mu}$.

This representation-theoretic problem can be reframed in a geometric fashion, due to the geometric Satake equivalence, in the study of cyclic convolution varieties $\cyc$. These are realized as (usually reducible) subvarieties in a product of affine Grassmannians $(\op{Gr}_G)^3$, and are remarkable for the following reason: the number of irreducible components of $\op{Gr}_{G ,c(\lambda, \mu, \nu)}$ of maximal dimension, known to be $\langle \rho, \lambda+\mu+\nu \rangle$, is precisely the multiplicity $m^{\nu^*}_{\lambda, \mu}$, where $\nu^*$ is the dual weight $-w_0\nu$.

Despite this remarkable property, the geometric structure of cyclic convolution varieties remains challenging and opaque in some respects. Even producing a point inside of a given convolution variety $\op{Gr}_{G, (\lambda, \mu, \nu)}$ can be highly nontrivial. Moreover, many geometric spaces which appear in geometric representation theory are endowed with stratifications which enable one to understand their geometry (consider flag varieties, spherical varieties, Springer fibers, and Schubert or affine Schubert varieties, to name a few). In contrast with this, understanding stratifications of cyclic convolution varieties remains an open problem. In \cite{Kiers}, it was shown that certain components of special cyclic convolution varieties are closures of diagonal left $G(\mathcal{O})$-orbits. In this work, we do the same for a new class of cyclic convolution varietes. However, we also produce an example of a maximal, irreducible component of a cyclic convolution variety which is \textit{not} a closure of a $G(\mathcal{O})$-orbit. We believe this is the first time in the literature that such an example appears. \\ 

We now provide the precise formulation of our main result. For notations, see Section 2. 
Let $G=PGL_{n+1}(\mathbb{C})$, with Langlands dual group $G^{\vee}=SL_{n+1}(\mathbb{C})$. Suppose that $\beta$ is a positive root in ${\Phi}_G$ and coweights $\lambda, \mu \in X_*(T)$ satisfy 

\begin{enumerate}
	\item[(1)] $\lambda+\mu-N\beta^\vee$ is dominant.
	\item[(2)] If $\langle \alpha_i,\lambda\rangle<N$ or $\langle\alpha_i,\mu\rangle < N$, then $\beta^\vee-\alpha_i^\vee \not \in \Phi_{G^\vee} \sqcup \{0\}$.
\end{enumerate}

We then produce irreducible components of a certain $\cyc$ as given by the following theorem.

\begin{theorem} \label{geom}
Let $\lambda, \mu \in X_{*}(T)$ and $\beta \in \Phi_{G}$ satisfy the conditions (1) and (2) above, and set $\nu=-w_0(\lambda+\mu-N\beta^\vee)$. Then 
\begin{enumerate}
\item[(i)] There exists a point $\xi \in \op{Gr}_{G,c(\lambda, \mu, \nu)}$ such that $\overline{G(\mathcal{O}).\xi}$ has maximal dimension $\langle \rho, 2 \lambda+2\mu-N\beta^{\vee} \rangle = \langle \rho, \lambda+\mu+\nu\rangle$.

\item[(ii)] Moreover, if $\beta$ is not a simple root, there exist two such disjoint $G(\mathcal{O})$-orbits. 

\end{enumerate}
\end{theorem}

Closures of $G(\mathcal{O})$-orbits are always irreducible, as $G(\mathcal{O})$ is connected. Thus we get the following corollary, strengthening a result of Kumar \cite{K} and Wahl \cite{W}. 

\begin{corollary} \label{multiplicity}
Let $N\ge 1$. For $G^\vee = SL_{n+1}(\C)$, let $\beta^\vee$ be a positive root and $\lambda, \mu$ dominant weights satisfying the conditions $(1)$ and $(2)$ as above.
Then $V(\lambda+\mu-N\beta^\vee)$ appears in $V(\lambda)\otimes V(\mu)$. And, if $\beta^\vee$ is not a simple root, then $V(\lambda+\mu-N\beta^\vee)$ appears with multiplicity at least $2$.
\end{corollary}

The layout of the paper is as follows: in Section 2 we introduce notation and provide some basic results on the geometry of cyclic convolution varieties, as well as the relevant representation theory. In Section 3 we construct the point $\xi$ belonging to $\cyc$. In Section 4 we find the dimension of the $G(\mathcal{O})$-orbit by passing to the tangent space. In Section 5 we further produce a second point $\tilde\xi$ with the same properties if $\beta$ is not a simple root. Lastly in Section 6 we examine the failure of the orbit method, showing in general that not all irreducible components are closures of $G(\mathcal{O})$-orbits on points.

\section{Preliminaries}
\subsection{Notation and Conventions}
Let $G$ be a reductive, connected algebraic group over $\C$, with a choice of maximal torus $T$ and Borel $B$, and set of roots $\Phi_{G}$. We will write $\Phi_G^+$ for the set of positive roots, and denote the simple roots with respect to the $T$-action on $B$ by $\{\alpha_i\}$. Let $G^\vee$ be the associated Langlands dual group with maximal torus and Borel subgroups $T^\vee$ and $B^\vee$, respectively, and set of roots $\Phi_{G^\vee}$. The associated Weyl group for $G$ and $G^\vee$ is denoted by $W$, with longest element $w_0$. 
We write $X_*(T) = \op{Hom}(\C^*,T)$ for the coweight lattice of $T$ and $X^*(T) = \op{Hom}(T,\C^*)$ for the weight lattice of $T$. We use $\{\varpi_i\} \subset X_*(T)$ to denote the basis of coweights dual to the simple roots $\{\alpha_i\}$; i.e., $\langle\alpha_i,\varpi_j\rangle = \delta_{i,j}$. If $\lambda \in X_*(T)$ is a dominant coweight, then it can be regarded as a dominant weight in $X^*(T^{\vee})$. We write $V(\lambda)$ to denote the finite-dimensional representation of $G^{\vee}$ of highest weight $\lambda$. We will also make use of $\rho=\frac{1}{2} \sum_{\alpha \in \Phi^+_G} \alpha$, the half sum of positive roots of $G$. We emphasize that while we work with $G^{\vee}$-representations, our notation is oriented around $G$: $\lambda, \mu, \nu, \varpi_i$ will be coweights of $T \subset G$, $\alpha_i, \beta$ will be roots of $T$. 
For a root $\beta$, we denote by $s_{\beta}\in W$ the reflection given by $\beta$. 
We also choose a pinning, denoted by $x_{\alpha}$, which gives embeddings $x_{\alpha}: \mathbb{G}_a \rightarrow G$ for each root $\alpha$, satisfying certain compatibilities.

\subsection{Affine Grassmannians and cyclic convolution varieties} Set $\mathcal{K} = \C((t))$ and $\mathcal{O} = \C[[ t ]]$. Then the affine Grassmannian $\op{Gr}_G$ associated to $G$ is given by the quotient $\op{Gr}_G:=G(\mathcal{K})/G(\mathcal{O})$. For every coweight $\lambda: \C^\times \to T$, there is an induced element $t^\lambda \in G(\mathcal{K})$. By the Cartan decomposition, the cosets $[\lambda]:= t^\lambda G(\mathcal{O})$ for dominant integral coweights $\lambda$  give a complete set of representatives for the left-$G(\mathcal{O})$ orbits inside the affine Grassmannian. We use this to define the \emph{distance} (or \emph{relative position}) $d(L_1, L_2)$ for points $(L_1, L_2)$ in $(\op{Gr}_G)^2$ to be the unique dominant coweight $\lambda$ such that, as left-$G(\mathcal{K})$ orbits in $(\op{Gr}_G)^2$, we have 
$$
G(\mathcal{K}).(L_1, L_2) = G(\mathcal{K}).([0],[\lambda])
$$
and write $d(L_1, L_2)=[\lambda]$. Following 
\cite{H}*{\S 2}, \cite{Kam}*{\S1},
we define the cyclic convolution variety $\op{Gr}_{G,c(\vec\lambda)}$ for a collection of dominant coweights $\vec{\lambda}=(\lambda_1, \dots \lambda_s)$ as 
$$
\op{Gr}_{G,c(\vec\lambda)}: = \left\{(L_1,\hdots,L_s)\in (\op{Gr}_G)^s \mid  L_s = [0], d(L_{i-1},L_i) = [\lambda_i]~ \forall i   \right\};
$$
here $L_0:=L_s$.  It is a finite-dimensional, complex algebraic variety whose dimension is always at most $\langle \rho, \sum \lambda_i\rangle$. Observe that it has a natural diagonal action of $G(\mathcal{O})$ on the left, since $g[0] = [0]$ for $g\in G(\mathcal{O})$ and $d(x,y) = d(gx, gy)$ for any $g\in G(\mathcal{K})$.

\subsection{Cyclic convolution varieties and root components} 
Recall the tensor decomposition problem of determining the components of the $G^\vee$-module
$$
V(\lambda) \otimes V(\mu) \simeq \bigoplus V(\nu_i)^{\oplus m_{\lambda, \mu}^{\nu_i}},
$$
as in the introduction. This is a well-studied problem, with various algebraic, geometric, and combinatorial methods developed to determine specific components $V(\nu_i)$; see \cite{K2} for an overview of this topic. One example of a family of components which will be our primary interest is the \emph{root components}, first constructed by Kumar. 

\begin{theorem}[\cite{K}] \label{original}
Let $G^\vee$ be a semisimple, simply-connected algebraic group over $\C$. Suppose that $\beta^\vee \in \Phi_{G^\vee}$ is a positive root of $G^\vee$ and $\lambda, \mu$ are dominant weights of $T^\vee$ satisfying 
\begin{enumerate}
    \item[(1')] $\lambda+\mu-\beta^\vee$ is dominant.
    \item[(2')] If $\langle\alpha_i,\lambda\rangle=0$ or $\langle\alpha_i,\mu\rangle=0$, then $\beta^\vee-\alpha_i^\vee \not \in \Phi_{G^\vee} \sqcup \{0\}$. 
\end{enumerate}
Then $V(\lambda+\mu-\beta^\vee)\subset V(\lambda)\otimes V(\mu)$.
\end{theorem}

Existence of irreducible components of this form was originally conjectured by Wahl \cite{W} and proven by him in the case of $G^\vee=SL_{n+1}(\C)$. We remark that Theorem \ref{original} has the following immediate corollary (cf. \cite{W}*{Theorem 6.5}), which is the form of root components that we consider.

\begin{corollary}\label{cor1}
Suppose $G^\vee$ has no component of type $G_2$. Let $N\ge 1$, and suppose that $\beta^\vee \in \Phi_{G^\vee}$, $\lambda, \mu$ as above satisfy
\begin{enumerate}
    \item[(1)] $\lambda+\mu-N\beta^\vee$ is dominant.
    \item[(2)] If $\langle \alpha_i,\lambda\rangle<N$ or $\langle\alpha_i,\mu\rangle < N$, then $\beta^\vee-\alpha_i^\vee \not \in \Phi_{G^\vee} \sqcup \{0\}$. 
\end{enumerate}
Then $V(\lambda+\mu-N\beta^\vee)\subset V(\lambda)\otimes V(\mu)$.
\end{corollary} 

\begin{proof}
Let $\rho_\beta$ be the dominant coweight of $T$ given by 
$$
\rho_\beta:= \sum_{i: \beta^\vee-\alpha_i^\vee\in \Phi_{G^\vee} \sqcup \{0\}} \varpi_i.
$$
That is, $\rho_\beta$ is the minimal dominant coweight satisfying condition (2'). 
By assumption (2), $\lambda = N\rho_\beta+\lambda'$ and $\mu = N\rho_\beta+\mu'$ for suitable \emph{dominant} coweights $\lambda',\mu'$. Since $G^\vee$ is at most doubly-laced, $2\rho_\beta- \beta^\vee$ is dominant. By Theorem \ref{original}, 
$$
V(2\rho_\beta - \beta^\vee)\subset V(\rho_\beta)\otimes V(\rho_\beta);
$$
therefore by scaling,
$
V(2N\rho_\beta - N\beta^\vee)\subset V(N\rho_\beta)\otimes V(N\rho_\beta).
$
Finally, we always have $V(\lambda'+\mu')\subseteq V(\lambda')\otimes V(\mu')$. By additivity of tensor product decompositions, 
$$
V(2N\rho_\beta-N\beta^\vee + \lambda'+\mu')\subset V(\lambda'+N\rho_\beta)\otimes V(\mu'+N\rho_\beta),
$$
as desired. 
\end{proof}

Returning now to cyclic convolution varieties, it is known via the geometric Satake equivalence (\cite{L}, \cite{G}, \cite{BD}, \cite{MV}) that the number of irreducible components of $\op{Gr}_{G,c(\vec\lambda)}$ which attain the maximal dimension $\langle \rho, \sum \lambda_i\rangle$ is equal to
$$
\dim (V(\lambda_1)\otimes\cdots\otimes V(\lambda_s))^{G^\vee};
$$
see also \cite{H}*{Proposition 3.1}. In fact, these irreducible components give a canonical basis of the latter vector space, and therefore irreducible components of maximal possible dimension of $\op{Gr}_{G,c(\vec\lambda)}$ are associated to components $V(-w_0(\lambda_s))$ of $V(\lambda_1) \otimes \cdots \otimes V(\lambda_{s-1})$. 

As a consequence of the root components of Kumar, we know that for dominant weights $\lambda$, $\mu$ and positive root $\beta^\vee$ of $G^\vee$ satisfying the conditions of Corollary \ref{cor1}, there must be irreducible components of $\op{Gr}_{G,c(\lambda, \mu, \nu)}$ of dimension $\langle \rho, \lambda+\mu+\nu \rangle$, where $\nu=-w_0(\lambda+\mu-N\beta^\vee)$. The content of Theorem \ref{geom}, and the three following sections, is an explicit construction of such components in the case of $G=PGL_{n+1}(\C)$.

\begin{example}
At this point we will pick up a running example for concreteness. On the representation-theoretic side, we work with $G^\vee=SL_5(\mathbb{C})$. We take $\beta=\alpha_2+\alpha_3$ a positive root of $G=PGL_5(\C)$, and take $\lambda=\varpi_2+\varpi_3$ and $\mu=\varpi_1+\varpi_2+\varpi_3+\varpi_4$. Lastly let us take $N=1$.  Then $\lambda+\mu-\beta^\vee=2\varpi_1+\varpi_2+\varpi_3+2\varpi_4$ is dominant. Moreover, $\langle \alpha_1, \lambda \rangle=\langle \alpha_4, \lambda \rangle=0$, and we see that neither $\beta^\vee-\alpha_1^\vee$ nor $\beta^\vee-\alpha_4^\vee$ is an element of $\Phi_{G^\vee}$, satisfying condition (2) of Theorem \ref{geom}. Condition (2) holds trivially for $\mu$. 
\end{example}

\section{A Good Point In $\cyc$}
Let $\lambda, \mu, \beta, $ and $\nu=-w_0(\lambda+\mu-N\beta^\vee)$ be as in the hypothesis of Theorem \ref{geom}. In this section we consider a point of $(\op{Gr}_G)^3 $ and prove that it is contained in the variety $\cyc$. 

By way of motivation, recall the well-known identity
\begin{align}\label{ide}
t^{-N\beta^\vee}=x_\beta(-t^{-N})x_{-\beta}(t^N)s_\beta^{-1}x_{-\beta}(t^{-N})
\end{align}
valid for $G$ of any type with a choice of pinning. 
A reasonable approach to finding points inside $\cyc$, then, is to examine points of the form
$$
([\lambda], g[\lambda+\mu-N\beta^\vee],[0])
$$
for some choice of $g\in G(\mathcal{O})$ making use of the expression (\ref{ide}). Clearly, $d([0],[\lambda]) = \lambda$. Moreover by acting on the left by $t^{-\lambda-\mu+N\beta^{\vee}}g^{-1} \in G(\mathcal{K})$, we see $d(g[\lambda +\mu-N \beta^{\vee}],[0])=d([0],[-\lambda-\mu+N\beta^{\vee}])$. This last point is not given by a dominant weight, but its dominant translate is precisely $-w_0(\lambda+\mu-N\beta^{\vee})=\nu$, so $d(g[\lambda+\mu-N\beta^\vee],[0])=\nu$. It remains to compute  $d([\lambda],g[\lambda+\mu-N\beta^{\vee}])=d([0],t^{-\lambda}gt^{\lambda}t^{-N\beta^{\vee}}[\mu])$.
\bigskip

Before proceeding, we record some general identities which will be of use in the following computation. 
Let $(g_1g_2)=g_1g_2g_1^{-1}g_2^{-1}$ be the commutator. We have the following lemma: 
\begin{lemma}\label{sberg} (\cite{Stein}, Chapter 3)
For arbitrary roots $\gamma_1, \gamma_2$ such that $\gamma_1+\gamma_2 \neq 0$, we have $(x_{\gamma_1}(a)x_{\gamma_2}(b))=\prod_{i \gamma_1+j \gamma_2 \in \Phi} x_{i\gamma_1+j\gamma_2}(c_{ij}a^ib^j)$ for $i, j \in \mathbb{Z}^+$ and constants $c_{ij}$ which depend on $\gamma_1$ and $\gamma_2$.
\end{lemma}

Since we will be working in Type $A$, root strings are of length at most two, the $c_{ij}$ are $\pm 1$, and we will rewrite this lemma as the commutator rule

\[x_{\gamma_2}(a)x_{\gamma_1}(b)=x_{\gamma_1}(b)x_{\gamma_2}(a)x_{\gamma_1+\gamma_2}(\pm ab)\]
 provided $\gamma_1 + \gamma_2$ is a root; otherwise, $x_{\gamma_1}(a)$ and $x_{\gamma_2}(b)$ commute. 

Moreover we have the following simple identity:

\begin{lemma}\label{simpity}
	$t^{\lambda}x_{\alpha}(a)t^{-\lambda}=x_{\alpha}(at^{\langle \alpha, \lambda \rangle})$.
\end{lemma}

We return to the examination of relative position of $([\lambda], g[\lambda+\mu-N\beta^{\vee}])$, which is the same as the relative position of $([0],t^{-\lambda}gt^{\lambda}t^{-N \beta^{\vee}}[\mu])$. The relative position will be $\mu$ so long as $t^{-\lambda} g t^{\lambda} t^{-N \beta^{\vee}}[\mu]$ is the same coset as  $g'[\mu]$ for some $g' \in G(\mathcal{O})$. We make a (semi-)judicious choice; let $g=x_{\beta}(t^{\langle \beta, \lambda \rangle -N})$, and moreover make the substitution (\ref{ide}) for $t^{-N\beta}$. We then have 

\begin{align*}
t^{-\lambda}gt^{\lambda}t^{-N\beta}[\mu] & = t^{-\lambda}gt^\lambda \left( x_\beta(-t^{-N})x_{-\beta}(t^N)s_\beta^{-1}x_{-\beta}(t^{-N}) \right) [\mu] \\
&= x_{\beta}(t^{-N})x_\beta(-t^{-N})x_{-\beta}(t^N)s_\beta^{-1}x_{-\beta}(t^{-N})[\mu] \\
&=x_{-\beta}(t^N)s_{\beta}^{-1} x_{-\beta}(t^{-N})[\mu].\\
\end{align*}

Lastly, note that $x_{-\beta}(t^{-N})t^{\mu}=t^{\mu}x_{-\beta}(t^{-N+\langle \beta, \mu \rangle})$. By condition (2) of Theorem \ref{geom} on $\lambda, \mu$ and $\beta^{\vee}$, $x_{-\beta}(t^{-N+\langle \beta, \mu \rangle}) \in G(\mathcal{O})$. Thus finally we obtain 

\begin{align*}
	x_{-\beta}(t^N)s_{\beta}^{-1}x_{-\beta}(t^-N)[\mu]=x_{-\beta}(t^N)s_{\beta}^{-1}[\mu].
\end{align*}

So we have that $d([\lambda], x_{\beta}(t^{\langle \beta, \lambda \rangle -N})[\lambda+\mu-N\beta^{\vee}])=d([0],x_{-\beta}(t^N)s_{\beta}^{-1}[\mu])$. Since $x_{-\beta}(t^N)s_{\beta}^{-1} \in G(\mathcal{O})$, we see that we have indeed produced a point in $\cyc$.

\begin{remark} Note that we have made use of both conditions $\langle \beta, \lambda \rangle -N \geq 0$ and $\langle \beta, \mu \rangle -N \geq 0;$ the first to ensure that if $t^{-\lambda}(g)t^\lambda=x_{\beta}(t^{-N})$ then $g \in G(\mathcal{O})$, and the second to note that $x_{-\beta}(t^{-N})[\mu]=[\mu]$. However, note that these are weaker conditions than condition (2) if $\beta$ is not simple. This may explain why the $G(\mathcal{O})$-orbit of this point is of insufficient dimension. 
\end{remark}

 Unfortunately, even in type $A$, if $\beta$ is not simple, then the $G(\mathcal{O})$-orbit of this point is not sufficient to produce a cycle of the correct  dimension in $\cyc$. Therefore we must modify our point. \\

We proceed now to name a ``good'' (i.e., its orbit dimension will be maximal) point in $\cyc$ for the case of $G=PGL_{n+1}(\C)$. As before, let $\{\alpha_1,\hdots,\alpha_n\}$ denote the standard choice of simple roots for $G$.  Then $\beta = \alpha_p + \alpha_{p+1} + \hdots +\alpha_q$ for some integers $1\le p\le q\le n$. We introduce the following type A specific notation: $\alpha_{i,j}=\alpha_i+\alpha_{i+1} + \dots + \alpha_j$. In type A, all positive roots are precisely of this form. 
Set
\begin{align}\label{xXxXx}
x = \prod_{i=p}^q x_{\alpha_{p,i}}(t^{\langle \alpha_{p,i},\lambda\rangle-N});
\end{align}
note that this product is independent of order of multiplication. Then $x$ is the unipotent matrix with $1$s on the diagonal and $t^{\langle \alpha_{p,i},\lambda\rangle-N}$ in the $(p,i+1)$ entry for each $p\le i\le q$ ($0$s elsewhere). Note that $\langle \alpha_{p,i},\lambda\rangle -N \ge \langle \alpha_p,\lambda \rangle-N\ge 0$ by condition (2) of Theorem \ref{geom}, so that indeed $x \in G(\mathcal{O})$.

\begin{prop}\label{point}
	The point $\xi = ([\lambda],x[\lambda+\mu-N\beta^\vee],[0])$ belongs to $\cyc$.
\end{prop}

\begin{proof}
	The computation uses the same ideas as those above.
	It is clear that $d([0],[\lambda]) = \lambda$ and $d(x[\lambda+\mu-N\beta^\vee],[0]) = \nu$, so it suffices to establish that $d([\lambda],x[\lambda+\mu-N\beta^\vee]) = \mu$. Rewriting slightly, we arrive at: \[d([\lambda],x[\lambda+\mu-N\beta^\vee])=d([0], t^{-\lambda}xt^{\lambda}t^{-N \beta^{\vee}}[\mu]).\]
	
	By Lemma \ref{simpity}, the expression $t^{-\lambda}xt^{\lambda}$ is equal to
	$$
	x' = \prod_{i=p}^q x_{\alpha_{p,i}}(t^{-N});
	$$
	leading us to examine 
	
	\[d([0], x't^{-N \beta^{\vee}}[\mu]).\]
	
	By identity (\ref{ide}),
	\begin{align*}
	x't^{-N\beta^\vee} [\mu] &= \left(\prod_{i=p}^{q}x_{\alpha_{p,i}}(t^{-N})\right)x_\beta(-t^{-N})x_{-\beta}(t^N)s_\beta^{-1}x_{-\beta}(t^{-N})[\mu] & ~ \\
	&= \left(\prod_{i=p}^{q-1}x_{\alpha_{p,i}}(t^{-N})\right)x_{-\beta}(t^N)s_\beta^{-1}x_{-\beta}(t^{-N})[\mu] & \textrm{since $\alpha_{p,q}=\beta$} \\
	&= \left(\prod_{i=p}^{q-1}x_{\alpha_{p,i}}(t^{-N})\right)x_{-\beta}(t^N)s_\beta^{-1}[\mu] & \textrm{since $x_{-\beta}(t^{-N})[\mu]=[\mu]$}. \\
	\end{align*}
	
	For $p \leq i \leq q$, observe that  \[\alpha_{p,i}+(-\beta)=(\alpha_p+ \dots +\alpha_{i})-(\alpha_p+ \dots \alpha_i + \dots +\alpha_q)=-\alpha_{i+1,q}\] is a root, so $x_{\alpha_{p,i}}(t^{-N})x_{-\beta}(t^N) = x_{-\beta}(t^N)x_{\alpha_{p,i}}(t^{-N})x_{-\alpha_{i+1,q}}(-1)$ (using Lemma \ref{sberg}). Furthermore, 
	the root subgroups 
	$x_{-\alpha_{i+1,q}}(\cdot)$ commute with $x_{\alpha_{p,j}}(\cdot)$ for any $i,j\le q-1$ since $-\alpha_{i+1,q}+\alpha_{p,j}$ is never a root under these conditions.
	
	Thus, moving $x_{-\beta}(t^N)$ all the way to the left, we obtain 
	
	\begin{align*} x' t^{-N\beta^{\vee}} &=
	\left(\prod_{i=p}^{q-1}x_{\alpha_{p,i}}(t^{-N})\right)x_{-\beta}(t^N)s_\beta^{-1}[\mu] \\ &= \underbrace{x_{-\beta}(t^N) \prod_{i=p}^{q-1} x_{-\alpha_{i+1,q}}(-1)}_{\text{ in $G(\mathcal{O})$}}\underbrace{\prod_{i=p}^{q-1} x_{\alpha_{p,i}}(t^{-N}) s_{\beta}^{-1}}_{y:=}[\mu],
	\end{align*}

	and we are reduced to showing $d([0],y[\mu]) = \mu$, with $y$ as indicated. One verifies that $s_\beta \alpha_{p,i} = \alpha_{p,i}-\beta = -\alpha_{i+1,q}$, so 
	$$
	y[\mu] = s_\beta^{-1} \prod_{i=p}^{q-1} x_{-\alpha_{i+1,q}}(t^{-N})[\mu].
	$$
	
	Remark that after conjugating by $t^{-\mu}$ we have  \[t^{-\mu}x_{-\alpha_{i+1,q}}(t^{-N})t^{\mu}=x_{-\alpha_{i+1,q}}(t^{\langle \alpha_{i+1,q}, \mu \rangle-N}).\]
	
	Finally we make use of the second condition placed on root components (2): If $\langle \alpha_i,\lambda\rangle<N$ or $\langle\alpha_i,\mu\rangle < N$, then $\beta^\vee-\alpha_i^\vee \not \in \Phi_{G^\vee} \sqcup \{0\}$.
	Since $\beta^\vee-\alpha^\vee_{q}$ is a root or zero, we must have $\langle \alpha_{i+1,q}, \mu \rangle \geq \langle \alpha_q, \mu \rangle \geq N$. This implies that the coset \[y[\mu] = s_\beta^{-1} \prod_{i=p}^{q-1} x_{-\alpha_{i+1,q}}(t^{-N})[\mu]=s_{\beta}^{-1}[\mu],\] and so finally \[d([0],y[\mu])=\mu\] as desired.
	
\end{proof}

\begin{example} In our running example, we have  $$\xi=([\varpi_2+\varpi_3],x[2\varpi_1+\varpi_2+\varpi_3+2\varpi_4],[0]),$$ where 
\[x=x_{\alpha_2}(t^{1-1})x_{\alpha_2+\alpha_3}(t^{2-1})=\begin{pmatrix}1 & 0 & 0 & 0 & 0 \\ 0 & 1 & t^0 & t^1 & 0 \\ 0 & 0 & 1 & 0  & 0 \\ 0 & 0 & 0 &1 & 0 \\ 0 & 0 & 0 & 0 & 1\\ \end{pmatrix}.\]
\end{example}

\section{The dimension of $G(\mathcal{O})\xi$}
In this section we will verify that the orbit $G(\mathcal{O})\xi\subseteq \cyc$ has dimension equal to $\langle \rho,2\lambda+2\mu-N\beta^\vee\rangle$. We observe that $G(\mathcal{O})\xi \simeq G(\mathcal{O})/\op{Stab}(\xi)$ as varieties. The subgroup $\op{Stab}(\xi)$ is the intersection of the stabilizers of the three elements comprising $\xi$: $\op{Stab}_{G(\mathcal{O})}([\lambda]) \cap \op{Stab}_{G(\mathcal{O})}(x[\lambda+\mu-N \beta^{\vee}]) \cap \op{Stab}_{G(\mathcal{O})}([0])$. Since we are interested in computing the dimension, we may pass to the tangent space of $G(\mathcal{O})/\op{Stab}(\xi)$ at the identity. Thus we wish to investigate the dimension of  the vector space
$$
\mathcal{T}=\dfrac{\mathfrak{g}(\mathcal{O})}
{\op{Ad}_{t^\lambda} \mathfrak{g}(\mathcal{O})\cap \op{Ad}_{x t^{\lambda +\mu -N\beta^\vee}}\mathfrak{g}(\mathcal{O}) \cap\mathfrak{g}(\mathcal{O}) }.
$$

Set $V:= \mathfrak{g}(\mathcal{O})\cap \op{Ad}_{t^\lambda} \mathfrak{g}(\mathcal{O})$ and $W:={\mathfrak{g}(\mathcal{O})\cap \op{Ad}_{t^\lambda} \mathfrak{g}(\mathcal{O})\cap \op{Ad}_{x t^{\lambda +\mu -N\beta^\vee}}\mathfrak{g}(\mathcal{O})}$. It is well-known that $\dim \mathfrak{g}(\mathcal{O})/V = \langle 2\rho, \lambda\rangle$, so from the short exact sequence
$$
0\to V/W \to \mathcal{T} \to \mathfrak{g}(\mathcal{O})/V \to 0
$$
of vector spaces, we see it is sufficient to verify that $\dim V/W = \langle \rho, 2\mu-N\beta^\vee\rangle$. The remainder of this section is devoted to this proof. 

\begin{prop}\label{dim}
	$\dim V/W = \langle 2\rho, \mu\rangle - N\langle \rho, \beta^\vee\rangle$.
\end{prop}

Here is our plan for the proof: we will describe the subspace $W\subseteq V$ explicitly by the vanishing of certain linear equations; this will afford us with a description of the quotient $V/W$, whose dimension we will then calculate. 

Let $v\in V$ be arbitrary. Since $v\in \mathfrak{g}(\mathcal{O})$ we may express $v$ as a matrix: 
$$
v = \left[\begin{array}{ccc}
v_{1,1} & \hdots & v_{1,n+1}\\
\vdots & & \vdots \\
v_{n+1,1} & \hdots & v_{n+1,n+1}
\end{array}
\right],
$$
with each $v_{i,j}\in \mathcal{O}$ and $v_{n+1,n+1} = -\sum_{i=1}^n v_{i,i}$. 
Since we will be working in this context (type A, with matrix coordinates) for the entirety of this section, we make the following notational convention.

\begin{notation}
	Let $\varepsilon_{i,j}$ be the positive root $\alpha_i+\hdots+\alpha_{j-1}$ whenever $i<j$. Therefore $\varepsilon_{i,j}+\varepsilon_{j,k} = \varepsilon_{i,k}$ for any $i<j<k$.
\end{notation}

The stipulation $v\in t^\lambda \mathfrak{g}(\mathcal{O})t^{-\lambda}$ means, for all $1\le i<j\le n+1$, 
$$
t^{\langle \varepsilon_{i,j},\lambda\rangle}\big|v_{i,j}.
$$

(Our $\varepsilon_{i,j}$ notation avoids the ugly but equivalent formulation $t^{\langle \alpha_{i,j-1},\lambda\rangle}\big|v_{i,j}$.) On the other hand, for all $i \geq j$, we merely require $val(v_{i,j}) \geq 0$ (here $val$ is the valuation map $\mathcal{K}\to \mathbb{Z}\cup \{\infty\}$ sending $t^ku\mapsto k$ if $u$ is a unit in $\mathcal{O}$ and $0\mapsto \infty$).
These conditions completely characterize elements of $V$.

Now, $v\in W$ if and only if $v\in V$ and $x^{-1}vx \in t^{ \nu^{*}}\mathfrak{g}(\mathcal{O})t^{- \nu^*}$, where $ \nu^* = \lambda+\mu-N\beta^\vee$. The crux of what follows is to explicitly write $u=x^{-1}vx = (u_{i,j})$ in matrix coordinates; then we can check containment in $t^{\nu^*}\mathfrak{g}(\mathcal{O})t^{- \nu^*}$ coordinate-wise. 

For $i<j$, consider the coordinate $u_{i,j}$, which corresponds to the root subgroup $x_{\alpha_{i,j-1}}$. Clearly the value of $u_{i,j}=(x^{-1}vx)_{(i,j)}$ depends on the relationship between $\alpha_{i,j-1}$ and $\beta$. (The cases $i\ge j$ will not need to be examined as closely: see Remark \ref{ignore}.)

To be precise, in root language there are the following four possibilities for a positive root:

\begin{enumerate}
	\item  A root $\alpha$ such that $\alpha-\alpha_{p,i} \notin \Phi \cup \{0\}$ for all $i$.
	\item A root $\alpha$ such that $\alpha-\alpha_{p,i} \in \Phi $ for multiple $i$'s, but $\alpha$ is not one of the $\alpha_{p,i}$.
	\item A root $\alpha$ such that $\alpha - \alpha_{p,i} \in \Phi$ for exactly one $i$, but $\alpha$ is not one of the $\alpha_{p,i}$.
	\item The root $\alpha$ is one of the $\alpha_{p,i}$.
\end{enumerate}

Translating from the language of root subgroups to matrix coordinates, these four cases are written in order below:
$$
u_{i,j} =\left\{
\begin{array}{cc}
v_{i,j}, &  j\le p \text{ or }j>q+1,i\ne p \\
v_{i,j}-\sum_{k=p}^{q} a_{k+1} v_{k+1,j}, & j\le p\text{ or }j>q+1, i=p\\
v_{i,j} + a_{j}v_{i,p}, & p<j\le q+1, i\ne p \\
v_{i,j} -\sum_{k=p}^q a_{k+1}v_{k+1,j} + a_j\left(v_{p,p} -\sum_{k=p}^q a_{k+1}v_{k+1,p}\right) , & p<j\le q+1,i=p
\end{array}
\right.
$$
where $a_j = t^{\langle \varepsilon_{p,j}, \lambda \rangle-N}$ for $p<j\le q+1$.

 This particular breakdown into 4 cases is a pleasant feature of working in type $A$; in the other types there would be more cases to consider.

\begin{example}
	In our running example of $A_4$ with $\beta=\alpha_2+\alpha_3 \in \Phi_{G}$ we provide some explicit examples of the different cases. 
	\begin{enumerate}
		\item Case 1: $\alpha_1,  \alpha_4, \alpha_{3,4}, \alpha_{1,4} $
		\item Case 2: $\alpha_{2,4}$
		\item Case 3: $\alpha_3, \alpha_{1,2}, \alpha_{1,3}$
		\item Case 4: $\alpha_2, \alpha_{2,3}$.
	\end{enumerate}

	The conjugated matrix is

	$$
	\hspace{-0.5in}\left[
	\begin{array}{ccccc}
	v_{11}  & v_{12}    &   v_{13} + v_{12}    &  v_{14}+  v_{12}t   &     v_{15} \\
	v_{21} - v_{31}    -v_{41}t  &   v_{22} - v_{32}   -v_{42}t &  v_{23}-v_{33}-v_{43}t  & v_{24}-v_{34}-v_{44}t & v_{25} - v_{35}  -v_{45}t \\
	&     &   +(v_{22}-v_{32}-v_{42}t)   &   +t(v_{22}-v_{32}-v_{42}t)   & \\
	v_{31}    &      v_{32}    &     v_{33} + v_{32}    &     v_{34} + v_{32}t     &     v_{35}\\
	v_{41}  &      v_{42}     &    v_{43} + v_{42}   &     v_{44}  + v_{42}t   &     v_{45}\\
	v_{51}   &      v_{52}   &     v_{53} + v_{52}    &    v_{54}+ v_{52}t  &    v_{55}
	\end{array}
	\right].
	$$
	
\end{example}

	We tabulate the minimal possible valuations of the $u_{i,j}$. 
	

\begin{listy} \label{list1}	
	\begin{enumerate}[label = (\Alph*)]
		
			\item {\bf Case $i \geq j$.} 
			
			If $i \geq j$, then $u_{i,j}$ is a polynomial function in $t$ and $v_{i',j'}$ such that $i' \geq j'$, in which $v_{i,j}$ appears as a summand. The only condition on $v_{i,j}$ for $i \geq j$ is that $val(v_{i,j}) \geq 0$. Thus the minimum possible valuation for all such matrix coefficients is $0$.
		
\bigskip\noindent 
In the remaining cases, we therefore assume $i<j$. 
\bigskip
		
		\item {\bf Case $j\le p$ or $j>q+1$, $i\ne p$.}
		
		Since $u_{i,j}=v_{i,j}$, we know that $u_{i,j}$ is always divisible by $t^{\langle \varepsilon_{i,j}, \lambda \rangle}$.

		\item {\bf Case $i=p$, $j>q+1$}

		Here $u_{i,j} = v_{i,j}-\sum_{k=p}^{q} a_{k+1} v_{k+1,j}$. 

			We have $v_{i,j}$ is divisible by $t^{\langle \varepsilon_{i,j}, \lambda \rangle}$. Recall that $a_l=t^{\langle \varepsilon_{p,l}, \lambda \rangle -N}$. Moreover we know that $v_{k+1,j}$ is divisible by $t^{\langle \varepsilon_{k+1,j}, \lambda \rangle}$, since all of the $v_{k+1,j}$ appearing in this sum have $k+1 \leq j$. Thus we find that each term $a_{k+1}v_{k+1,j}$ is divisible by $t^{\langle \varepsilon_{p, k+1}, \lambda \rangle -N}t^{\langle \varepsilon_{k+1, j}, \lambda \rangle }=t^{\langle \varepsilon_{p,j}, \lambda \rangle -N}$. Since $i=p$ in this case, $v_{i,j}$ is divisible by $t^{\langle \varepsilon_{p,j}, \lambda \rangle}$, and so $u_{i,j}$ is divisible by $t^{\langle \varepsilon_{i,j}, \lambda \rangle -N}$.

		\item {\bf Case $p<j\le q+1$, $i\ne p$} 
		\begin{enumerate}
			 In this case $u_{i,j} = v_{i,j}+a_{j}v_{i,p}$. 
			 
			 \item $i >p$. As before $v_{i,j}$ is divisible by $t^{\langle \varepsilon_{i,j},\lambda \rangle}$. If $i>p$, $v_{i,p}$ merely needs to have nonnegative valuation so $a_jv_{i,p}$ is divisible by 
			$$
			t^{\langle \varepsilon_{p, j}, \lambda \rangle -N}.
			$$
			Note that, since $\beta^\vee - \alpha^\vee_p \in \Phi_{G^\vee} \sqcup \{0\}$, we have
			$$
			\langle \varepsilon_{p,j},\lambda\rangle -N \ge \langle \varepsilon_{p+1,j},\lambda\rangle \ge \langle \varepsilon_{i,j},\lambda\rangle,
			$$
			so if $i>p$ then $t^{\langle\varepsilon_{i,j},\lambda \rangle}\big|u_{i,j}$.
			\item $i<p$. Again $v_{i,j}$ is divisible by $t^{\langle \varepsilon_{i,j}, \lambda \rangle}$, and $a_j v_{i,p}$ is divisible by \[t^{\langle \varepsilon_{p,j}, \lambda \rangle -N}t^{\langle \varepsilon_{i,j}, \lambda \rangle}.\]
			
			\end{enumerate}

		\item {\bf Case $p=i<j\leq q+1$}
		
		Here $u_{i,j}$ is divisible by $t^{\langle \varepsilon_{p,j},\lambda\rangle -N}$, as we explain: the labels indicate lower bounds on powers of $t$ dividing each term below. 
		
		$$
		u_{p,j} = \underbrace{v_{p,j}}_{t^{\langle\varepsilon_{p,j},\lambda\rangle}}  -\sum_{k=p}^{j-2} \underbrace{a_{k+1}v_{k+1,j}}_{t^{\langle\varepsilon_{p,j},\lambda\rangle-N}} -\sum_{k=j-1}^q \underbrace{a_{k+1}v_{k+1,j}}_{t^{\langle\varepsilon_{p,k+1},\lambda\rangle-N}} + \underbrace{a_j}_{t^{\langle\varepsilon_{p,j},\lambda\rangle-N}} \left(v_{p,p} -\sum_{k=p}^q a_{k+1}v_{k+1,p}\right);
		$$
		note that if $k\ge j-1$ then $\langle \varepsilon_{p,k+1},\lambda\rangle \ge \langle \varepsilon_{p,j},\lambda\rangle$.

	\end{enumerate}
\end{listy}
With this result in hand, we have control over the minimum possible valuations for any matrix function $u_{ij}$ where $u=x^{-1}vx$ and $v \in V$. 

Our next task is to compare these minimum valuations with the minimum valuations imposed by demanding that $x^{-1}vx \in Ad_{t^{\nu^*}}(\mathfrak{g}(\mathcal{O}))$.

\begin{remark}\label{ignore}
Observe that no further condition needs to be applied to the coordinates $u_{i,j}$, $i\le j$ in order to demand that $x^{-1}vx\in Ad_{t^{\nu^*}}(\mathfrak{g}(\mathcal{O}))$. Therefore we may ignore case (A) in comparing the vector spaces $V$ and $W$. 
\end{remark}

In our description of minimum valuations above, cases (C), (D) and (E) depend on case (B) in the sense that if we demand that all matrix coefficients of the form (B) have higher valuations, then the minimum possible valuations in the remaining cases may change as well. For this reason we study case (B) first.

	Let $v \in V$, so that $x^{-1}vx$ has minimum valuation data as in List \ref{list1}. A necessary (but insufficient) condition for $x^{-1}vx \in W$ is that for all $u_{i,j}=(x^{-1}vx)_{i,j}$ of type (B), we must have \[t^{\langle \varepsilon_{i,j}, v^* \rangle}|u_{i,j}\] or 
	
	\begin{equation} \label{cond}
	 val(u_{i,j}) \geq \langle \varepsilon_{i,j}, \lambda+\mu-N \beta^{\vee} \rangle .
	\end{equation}
	
	This motivates the definition of the following map $\phi_1$.

	\begin{lemma}\label{lem1}
		Define a map 
		\begin{center}
			\begin{tikzcd}
				V\arrow[r,"\phi_1"] & \displaystyle\bigoplus_{\textnormal{ case (B)}} \dfrac{t^{\langle \varepsilon_{i,j},\lambda \rangle}\mathcal{O}}{(t^{\langle \varepsilon_{i,j},\nu^* \rangle})}\\
				v \arrow[r, maps to] & (u_{i,j}).
			\end{tikzcd}
		\end{center}
		Then $\phi_1$ is surjective. 
		
	\end{lemma}
	
	\begin{proof}
		For a given $i,j$ of case (B), set $v_{i,j} = t^{\langle \varepsilon_{i,j},\lambda\rangle}$ and all other $v_{i',j'}=0$. Then $v\in V$ and $\phi_1(v)$ generates the range as an $\mathcal{O}$-module. As $\phi_1$ is an $\mathcal{O}$-linear morphism of $\mathcal{O}$-modules, this establishes surjectivity. 
		\end{proof}

	Set $W':=\ker \phi_1$. Note that $W\subseteq W'$, but in general they are not equal. Note also that $W'$ is an $\mathcal{O}$-submodule of $V$.
	
	We see that if $v \in V$ is also in $ker(\phi_1)$ then we gain additional information about the minimal possible valuations on the matrix coordinates $u_{ij}=(x^{-1}vx)_{i,j}$ where $i,j$ falls into cases (C),(D),or (E). We record this additional information in the following lemma.
	
	\begin{lemma}\label{list2}
		Assume that $v \in V$ and that $\phi_1(v)=0$. Then $u_{i,j}=(x^{-1}vx)_{i,j}$ is divisible by $t^{\langle \varepsilon_{i,j}, \lambda \rangle}$ if $i,j$ falls in case (C) or (D), and $u_{i,j}$ is divisible by $t^{\langle \varepsilon_{i,j}, \lambda \rangle -N}$ if $u_{i,j}$ falls in case (E). Note that this increases the minimal possible valuations for cases (C) and (D); compare with the minimal possible valuations in List \ref{list1}.
	\end{lemma}

\begin{proof}
	\begin{enumerate}
		\item Case (C): Here $u_{i,j} = v_{i,j}-\sum_{k=p}^{q} a_{k+1} v_{k+1,j}$. 
		
		We know that $v_{i,j}$ is divisible by $t^{\langle \varepsilon_{i,j},\lambda\rangle}$. Since $\phi_1(v)=0$, we know that for each pair $i,j$ in case (B), we have $t^{\langle \varepsilon_{i,j}, \nu^* \rangle}|v_{i,j}$. In other words, each $v_{k+1,j}$ is divisible by $t^{\langle \varepsilon_{k+1,j}, \lambda+\mu-N\beta^\vee \rangle}$. If $k<q$, then 
		$$
		\langle \varepsilon_{k+1, j} ,\mu\rangle \ge \langle \alpha_{q},\mu\rangle\ge N
		$$
		by condition (2) of Theorem \ref{geom} and 
		$$
		\langle \varepsilon_{k+1,j}, -N\beta^\vee\rangle = 0,
		$$
		since necessarily $p<k+1$ and therefore $\varepsilon_{k+1,j}-\beta$ is not a root. 
		Otherwise, $k=q$ and 
		$$
		\langle \varepsilon_{q+1,j},\mu\rangle \ge 0;
		$$
		$$\langle \varepsilon_{q+1,j}, -N\beta^\vee\rangle = N. 
		$$
		
		Therefore in any case
		$$
		\langle \varepsilon_{k+1,j},\lambda+\mu-
		N\beta^\vee\rangle\ge \langle \varepsilon_{k+1,j},\lambda\rangle +N
		$$
		and each term $a_{k+1} v_{k+1,j}$ is divisible by 
		$$
		t^{\langle \varepsilon_{p,k+1},\lambda \rangle - N}\cdot t^{\langle \varepsilon_{k+1,j}  ,\lambda  \rangle+N} = t^{\langle \varepsilon_{p,j},\lambda\rangle },
		$$ 
		so the entire $u_{i,j}$ is divisible by $t^{\langle \varepsilon_{p,j},\lambda\rangle }$.
		
		\item Case (D): In this case $u_{i,j} = v_{i,j}+a_{j}v_{i,p}$.
		
		 As before $v_{i,j}$ is divisible by $t^{\langle \varepsilon_{i,j},\lambda \rangle}$. If $i>p$, $a_jv_{i,p}$ is divisible by 
		$$
		t^{\langle \varepsilon_{p,j},\lambda \rangle-N};
		$$
		note that 
		$$
		\langle \varepsilon_{p,j},\lambda\rangle -N \ge \langle \varepsilon_{p+1,j} ,\lambda\rangle \ge \langle \varepsilon_{i,j},\lambda\rangle,
		$$
		so if $i>p$ then $t^{\langle\varepsilon_{i,j},\lambda \rangle}\big|u_{i,j}$.
		
		On the other hand, if $i<p$, then 
		$$
		\langle \varepsilon_{i,p} ,\mu-N\beta^\vee\rangle \ge \langle \varepsilon_{i,p},-N\beta^\vee\rangle = N,
		$$
		so the condition $\phi_1(v)=0$ is sufficient to imply that $a_jv_{i,p}$ is divisible by 
		$$
		t^{\langle \varepsilon_{p,j},\lambda \rangle-N}\cdot t^{\langle \varepsilon_{i,p},\lambda \rangle+N} = t^{\langle \varepsilon_{i,j},\lambda \rangle}
		$$
		if $i<p$. 
		
		We conclude that $u_{i,j}$ is divisible by $t^{\langle \varepsilon_{i,j},\lambda\rangle}$ in either case.
		
		\item Case (E): in this case, demanding that $\phi_1(v)=0$ does not provide us with higher minimal possible valuations of these matrix entries. We nonetheless still have $t^{\langle \varepsilon_{i,j}, \lambda \rangle -N}|u_{i,j}$ by List \ref{list1}.
	\end{enumerate}	
\end{proof}	

\begin{summary}
	So far, we let $v \in V= \mathfrak{g}(\mathcal{O}) \cap Ad_{t^{\lambda}}(\mathfrak{g}(\mathcal{O}))$. After conjugating $ x^{-1}vx$, we obtain minimal possible valuations for all matrix entries $u_{i,j}=(x^{-1}vx)_{i,j}$. These are listed in List \ref{list1}. There are 5 different possibilities (A),(B),(C),(D),(E), depending on the matrix coordinate $i,j$. For our dimension comparison $dim V/W$, we may ignore all of the matrix coefficients below the main diagonal since the valuation conditions on these entries will always be trivially satisfied; this means that we may ignore case (A).
	
	The next step is to examine when $u=x^{-1}vx \in W$. As mentioned before, this adds no new constraints for matrix coordinates of case (A). So we next assume that the matrix coefficients in case (B) have sufficiently high valuations so that $u$ has a chance of being in $W$. This assumption on the matrix coefficients in case (B) will generally increase minimal possible  valuations of $u_{i,j}$ for $i,j$ in cases (C) and (D), but not (E). These new valuations are listed in Lemma \ref{list2}.
\end{summary}	

We may now proceed with a surjectivity lemma, and then finally with the dimension computation.
	\begin{lemma}\label{lem2}
		Define a map 
		\begin{center}
			\begin{tikzcd}
				W' \arrow[r,"\phi_2"] & 
				\displaystyle\bigoplus_{\textnormal{ case (C)}} \dfrac{t^{\langle \varepsilon_{p,j},\lambda \rangle}\mathcal{O}}{(t^{\langle \varepsilon_{p,j},\nu^* \rangle})} \oplus 
				\displaystyle\bigoplus_{\textnormal{ case (D)}} \dfrac{t^{\langle \varepsilon_{i,j},\lambda \rangle}\mathcal{O}}{(t^{\langle \varepsilon_{i,j},\nu^* \rangle})} \oplus 
				\displaystyle\bigoplus_{\textnormal{ case (E)}} \dfrac{t^{\langle \varepsilon_{p,j},\lambda \rangle-N}\mathcal{O}}{(t^{\langle \varepsilon_{p,j},\nu^* \rangle})} \\
				v \arrow[r,maps to] & ( u_{p,j}, \hspace{1in}  u_{i,j}, \hspace{1in}  u_{p,j}),
			\end{tikzcd}
		\end{center}
		notation as above. Then $\phi_2$ is well-defined and surjective, and $\ker \phi_2\simeq W$. 
	\end{lemma}
	
	\begin{proof}
		
		By definition, $\ker \phi_2 \simeq W$; appearing in this kernel means satisfying the valuation requirements for all possible matrix coefficients $u_{i,j}$. The map is well-defined by the divisibility considerations of cases (C) and (D), assuming (\ref{cond}) holds in case (B) (this is why we have created the space $W'$ and restricted our attention there). 
	
		An $\mathcal{O}$-module basis of the space on the right consists of matrices with a single nonzero entry $(t^n)_{i,j}$ where $u_{i,j}$ corresponds to one of the cases (C), (D) or (E), and where $n = \langle \varepsilon_{i,j}, \lambda \rangle $ when $u_{i,j}$ is in cases (C) and (D), and $n = \langle \varepsilon_{i,j}, \lambda \rangle -N $ when $u_{i,j}$ is in case (E). It is straightforward to produce elements of $W'$ which map to these basis elements under $\phi$.
		\begin{enumerate}
		\item Let $p,j$ be as in case (C). Set $v_{p,j} = t^{\langle \varepsilon_{p,j},\lambda\rangle}$ and all other $v_{i',j'}=0$. Then $v\in W'$ (the relevant $v_{i',j'}$s are $0$) and $\phi_2(v) = (\underbrace{t^{\langle \varepsilon_{p,j}\rangle}}_{\text{position $p,j$}},0,0)$.  
		
		\item Let $i,j$ be as in case (D). Set $v_{i,j} = t^{\langle \varepsilon_{i,j},\lambda\rangle}$ and $v_{i,p}=0$. If $i>p$, then set $v_{j,j} = 1$. Set all other $v_{i',j'}=0$. Therefore $v\in W'$ and $\phi_2(v) = (0,\underbrace{t^{\langle \varepsilon_{i,j},\lambda\rangle}}_{\text{position $i,j$}},0)$ (a cancellation occurs at position $p,j$ in the latter $0$ if $i>p$).
		
		\item Let $p,j$ be as in case (E). Set $v_{j,j}=1$ and all other $v_{i',j'}=0$. Then $v\in W'$ and $\phi_2(v) = (0,0,\underbrace{t^{\langle \varepsilon_{p,j}\rangle-N}}_{\text{position $p,j$}})$.
		\end{enumerate}
		We see that the standard $\mathcal{O}$-module generators of the range of $\phi_2$ are indeed in the image of $\phi_2$. Since $\phi_2$ is an $\mathcal{O}$-morphism, surjectivity follows. 
	\end{proof}

	We are finally in a position to prove Proposition \ref{dim}.
\begin{proof}[Proof of Proposition \ref{dim}]
	 The exact sequence 
	$$
	0\to W'/W \to V/W \to V/W' \to 0
	$$
	tells us that $\dim(V/W) = \dim (V/W')+\dim (W'/W)$. By construction,
	$$
	\dim(V/W') = \dim(\textnormal{Im}(\phi_1)) = \sum_{\textnormal{case (B)}} \langle \varepsilon_{i,j}, \nu^* - \lambda \rangle = \sum_{\textnormal{case (B)}} \langle \varepsilon_{i,j}, \mu-N\beta^\vee \rangle \\ 
	$$
	and similarly
	
	$$
	\begin{aligned}
	\dim(W'/W) = \dim(\textnormal{Im}(\phi_2)) = \sum_{\textnormal{ case (C)}} \langle & \varepsilon_{p,j},\mu-N\beta^\vee\rangle +
	\sum_{\textnormal{ case (D)}} \langle \varepsilon_{i,j},\mu-N\beta^\vee\rangle \\
	&+\sum_{\textnormal{ case (E)}} \left(\langle \varepsilon_{p,j},\mu-N\beta^\vee\rangle+N\right).
	\end{aligned}
	$$
	
	Note that, in case (E), the summation runs over $j$ such that $p < j \leq q+1$, which is $\langle \rho,\beta^\vee\rangle$-many terms. As each positive root of $G$ appears as an $\varepsilon_{i,j}$ in exactly one of our four cases (B),(C),(D),(E), we have in total that 
	$$
	\dim(V/W) = \langle 2\rho, \mu-N\beta^\vee \rangle +N\langle \rho,\beta^\vee\rangle = \langle 2\rho, \mu-N\beta^\vee \rangle + \langle \rho, N\beta^\vee \rangle = \langle \rho, 2\mu-N\beta^\vee \rangle.
	$$
\end{proof}

\begin{corollary}
	The dimension of the orbit $G(\mathcal{O}).\xi= G(\mathcal{O}).([\lambda], x[\lambda + \mu -N \beta^{\vee}],[0])$ is $\langle \rho, 2 \lambda+2 \mu-N \beta^{\vee} \rangle$. Thus the closure of this orbit is a maximal irreducible component of the cyclic convolution variety.
\end{corollary}

\section{A second point if $\beta$ is not simple}

So far, the combination of Propositions \ref{point} and \ref{dim} proves the existence of a maximal dimensional component as desired in Theorem \ref{geom}. In this section we produce a second component of maximal dimension, assuming $\beta$ is not simple. Recall that there is a distinguished order two automorphism $\sigma: PGL_n(\mathbb{C}) \rightarrow PGL_n(\mathbb{C})$ obtained by \[A \mapsto \tilde{w}_0 (A^{t})^{-1} \tilde{w}_0^{-1}\] where $\tilde{w}_0 \in N(T)$ is a representative of the longest word $w_0 \in N(T)/T$.

We construct the second point as follows. Since $\sigma$ is obtained from a diagram automorphism, we will abuse notation and also write $\sigma$ for the action on weights or coweights. We extend the action of $\sigma$ to $G(\mathcal{K})$ via $t \mapsto -t$. Since $\sigma$ preserves $G(\mathcal{O})$, $\sigma$ acts on $\op{Gr}_G$. Moreover, by letting $\sigma$ act on all factors, we obtain an isomorphism \[\sigma: 	\cyc \to \op{Gr}_{G,c(\sigma(\lambda),\sigma(\mu),\sigma(\nu))}.\]

By our construction in the above sections, we may produce a point \[\xi'= ([\sigma(\lambda)], x'[\sigma(\lambda+\mu-N \beta^{\vee})], [0]) \in \op{Gr}_{G,c(\sigma(\lambda),\sigma(\mu),\sigma(\nu))}.\] By Proposition \ref{dim}, the $G(\mathcal{O})$-orbit through this point is of the  correct dimension. Thus, the image \[\sigma(\xi')= ([\lambda], \sigma(x')[\lambda+\mu-N\beta^{\vee}],[0]) \in \cyc\] is a point such that the $G(\mathcal{O})-$orbit is of the correct dimension. 

\begin{lemma}
	Let $\tilde x =  \prod_{i=p}^q x_{\alpha_{i,q}}(t^{\langle \alpha_{i,q},\lambda\rangle-N})$. Then $G(\mathcal{O}).([\lambda], \tilde{x}[\lambda+\mu-N\beta^{\vee}],[0])=G(\mathcal{O}).\sigma(\xi') \subset \cyc$.
\end{lemma}

\begin{proof}
	While $\tilde{x} \neq \sigma(x')$, these matrices only differ by some negative signs. However, we may transform from $\tilde{x}$ to $\sigma(x')$ via right and left multiplication by elements of the form $\varpi_i(-1)$, and these are elements of $PGL_n(\mathcal{O})$. 
\end{proof}

We will use the notation $\tilde{\xi}=\sigma(\xi')$.

\begin{example}
	In our running example, $\sigma(\lambda)=\lambda$, $\sigma(\mu)=\mu$ and $\sigma(\beta^{\vee})=\beta^{\vee}$. Thus $x'=x$, and $\sigma(x')$ is the matrix 
	\[\begin{pmatrix}1 & 0 & 0 & 0 & 0 \\ 0 & 1 & 0 & t^1 & 0 \\ 0 & 0 & 1 & -t^0  & 0 \\ 0 & 0 & 0 &1 & 0 \\ 0 & 0 & 0 & 0 & 1\\ \end{pmatrix}.\]
	Also, $\tilde{x}$ is the matrix 
		\[\begin{pmatrix}1 & 0 & 0 & 0 & 0 \\ 0 & 1 & 0 & t^1 & 0 \\ 0 & 0 & 1 & t^0  & 0 \\ 0 & 0 & 0 &1 & 0 \\ 0 & 0 & 0 & 0 & 1\\ \end{pmatrix},\]
and \[\tilde{x}=\begin{pmatrix} 1 & 0 & 0 & 0 & 0 \\ 0 & 1 & 0&0 &0 \\ 0 & 0 & -1 & 0 & 0 \\ 0 & 0 & 0 & 1 & 0 \\ 0 & 0 & 0 & 0 & 1 \\ \end{pmatrix}	\begin{pmatrix}1 & 0 & 0 & 0 & 0 \\ 0 & 1 & 0 & t^1 & 0 \\ 0 & 0 & 1 & -t^0  & 0 \\ 0 & 0 & 0 &1 & 0 \\ 0 & 0 & 0 & 0 & 1\\ \end{pmatrix}\begin{pmatrix} 1 & 0 & 0 & 0 & 0 \\ 0 & 1 & 0&0 &0 \\ 0 & 0 & -1 & 0 & 0 \\ 0 & 0 & 0 & 1 & 0 \\ 0 & 0 & 0 & 0 & 1 \\ \end{pmatrix}.\]
\end{example}

What remains to be shown is that \[G(\mathcal{O}). \xi \cap G(\mathcal{O}).\tilde{\xi} = \emptyset,\] so that we really have produced two distinct irreducible components of $\cyc$.

\begin{prop}
	With all notation as above, $G(\mathcal{O})\xi \cap G(\mathcal{O})\tilde \xi = \emptyset$.
\end{prop}

\begin{proof}
	Assume for the sake of contradiction that $\tilde \xi = g \xi$ for some $g\in G(\mathcal{O})$. Immediately we recognize that $g$ also belongs to $t^{\lambda}G(\mathcal{O})t^{-\lambda}$ since $g\in \op{Stab}([\lambda])$. Write 
	$$
	g = \left[\begin{array}{ccc}
	g_{1,1} & \hdots & g_{1,n+1}\\
	\vdots & & \vdots \\
	g_{n+1,1} & \hdots & g_{n+1,n+1}
	\end{array}
	\right]
	$$
	as an invertible matrix in $PGL_{n+1}(\mathcal{O})$; so far we know $t^{\langle \varepsilon_{i,j},\lambda \rangle}\big| g_{i,j}$ whenever $i<j$ (once again $\varepsilon_{i,j} := \alpha_i+\hdots+\alpha_{j-1}$). 
	
	By assumption, $\tilde x^{-1} g x$ belongs to $t^{\nu^*} G(\mathcal{O}) t^{-\nu^*}$. 
	Let $h = \tilde x^{-1} g x$, and let $\{h_{i,j}\}$ be the matrix coordinates of $h$. Then 
	$$
	h_{i,j}=\left\{
	\begin{array}{cc}
	g_{i,j}, & i\not \in [p,q],j\not\in [p+1,q+1]\\
	g_{i,j} - g_{q+1,j}c_i, & i \in [p,q], j\not \in [p+1,q+1]\\
	g_{i,j} + g_{i,p}b_j, & i\not\in [p,q], j\in [p+1,q+1]\\
	g_{i,j} + g_{i,p}b_j - g_{q+1,j}c_i - g_{q+1,p}c_ib_j, & ~i\in [p,q], j\in [p+1,q+1]
	\end{array}
	\right.
	$$
	where $b_j = t^{\langle \varepsilon_{p,j},\lambda \rangle -N}$ and $c_i = t^{\langle \varepsilon_{i,q+1},\lambda \rangle -N}$. 
	
	\begin{lemma}
		For $q+1\le j\le n$, $g_{q+1,j}$ is divisible by $t$. 
	\end{lemma}
	\begin{proof}
		
		Let $j$ be strictly bigger than $q+1$ and $\le n$ (if possible). From above, $h_{q,j}= g_{q,j} - g_{q+1,j}t^{\langle \varepsilon_{q,q+1},\lambda \rangle-N}$.

		From the assumption $t^{\langle \varepsilon_{q,j},\lambda+\mu-N\beta^\vee\rangle} \big | h_{q,j}$, from the observation that $\langle \varepsilon_{q,j},\lambda+\mu-N\beta^\vee\rangle \ge \langle \varepsilon_{q,j},\lambda\rangle$, and from the knowledge that $t^{\langle \varepsilon_{q,j},\lambda\rangle}\big|g_{q,j}$, 
		we see that $t^{\langle \varepsilon_{q,j}\lambda,\rangle}$ should divide $g_{q+1,j}t^{\langle \varepsilon_{q,q+1},\lambda \rangle-N}$. This of course boils down to 
		$$
		t^{\langle \varepsilon_{q+1,j},\lambda\rangle+N}|g_{q+1,j},
		$$
		so at least $t\big| g_{q+1,j}$. 
		
		For $j=q+1$, we have
		$$
		h_{q,q+1} = g_{q,q+1}+t^{\langle \varepsilon_{p,q+1},\lambda\rangle-N}(g_{q,p} - g_{q+1,p}t^{\langle\varepsilon_{q,q+1},\lambda \rangle -N}) - g_{q+1,q+1}t^{\langle\varepsilon_{q,q+1},\lambda \rangle -N}.
		$$
		Observe that $t^{\langle \varepsilon_{q,q+1},\lambda\rangle}\big | g_{q,q+1}$ and $t^{\langle \varepsilon_{q,q+1},\lambda\rangle}\big | t^{\langle \varepsilon_{p,q+1},\lambda\rangle -N}$. Furthermore, our assumption implies $t^{\langle \varepsilon_{q,q+1},\lambda\rangle}\big | h_{q,q+1}$. So 
		$t^{\langle \varepsilon_{q,q+1},\lambda\rangle}\big | g_{q+1,q+1}t^{\langle \varepsilon_{q,q+1},\lambda\rangle-N}$; therefore $t$ divides $g_{q+1,q+1}$. 
	\end{proof}
	
	We now arrive at a contradiction: the matrix $g$ is not invertible. Indeed, the entries $g_{i,j}$ where $i\le q$ and $j\ge q+1$ are divisible by $t^{\langle \varepsilon_{i,j},\lambda \rangle}$, which is a nonzero power of $t$. By the preceding lemma, the entries $g_{q+1,j}$ for $j\ge q+1$ are also divisible by $t$. This forces the lower-right $(n-q-1)\times (n-q)$ submatrix to contain the pivot points for the last $n-q$ columns (assuming $g$ is invertible), which cannot happen. 
\end{proof}

\section{Failure of $G(\mathcal{O})$-orbit methods} \label{JoshEx} 

We produce an example of a maximal component in a cyclic convolution variety which is not a closure of a $G(\mathcal{O})$-orbit on a point. For this example, we take $G = PGL_3(\C)$, $G^\vee = SL_3(\C)$, and coweights $\lambda = \mu = \nu = 2\rho^\vee$, the sum of all positive roots of $G^\vee$: in this case $2\rho^\vee = 2\varpi_1+2\varpi_2 = 2\alpha_1^\vee+2\alpha_2^\vee$; also note that $-w_0\nu = \nu$. 

The multiplicity of $V(2\rho^\vee)$ inside $V(2\rho^\vee)\otimes V(2\rho^\vee)$ can be shown by direct computation to be $3$. However, the geometric construction of this present paper (with $N=2, \beta^\vee = \alpha_1^\vee+\alpha_2^\vee$) only gives two corresponding irreducible components of the cyclic convolution variety. In fact, the triple $(\lambda,\mu,\nu)$ is also a so-called \textit{PRV triple}, and the proof of the refined PRV conjecture in \cite{Kiers} finds again two irreducible components of the same variety. These pairs of components turn out to be the same. 

The remaining third component can be described as follows. First, for every choice of complex parameter $a\in \C\setminus\{0,1\}$, the triple 
$$
\xi(a):=\left( [2\rho^\vee], 
u(a)[2\rho^\vee],[0]\right)
$$
belongs to $\cyc$, where we define  
$$
u(a):= \left[
\begin{array}{ccc}
1 & t & a t^2\\ 
0 & 1 & t \\
0 & 0 & 1
\end{array}
\right].
$$
The only nontrivial requirement to check is that $d([2\rho^\vee],u(a)[2\rho^\vee]) = 2\rho^\vee$. This follows from the equation 
$$
t^{-2\rho^\vee}\left[\begin{array}{ccc}t^2 & -t & 1-a \\\\ -\frac{t}{a-1} & \frac{a}{a-1} & 0 \\\\ \frac{1}{a} & 0 & 0 \end{array}\right] t^{-2\rho^\vee} u(a) t^{2\rho^\vee} = 
\left[
\begin{array}{ccc}
1 & 0 & 0 \\\\
-\frac{t}{a-1} & 1 & 0 \\\\
\frac{t^2}{a} & \frac{t}{a} & 1
\end{array}
\right],
$$
which is valid whenever $a\ne 0, 1$. 

For any $a\ne 0,1$, the dimension of the orbit $G(\mathcal{O})\xi(a)$ is equal to $11$, which is calculated by the same procedure as earlier in this paper. This falls short of the highest possible dimension of $\langle \rho, 6\rho^\vee\rangle = 12$, but as $a$ varies over the $1$-dimensional family $\C\setminus \{0,1\}$, the family of $G(\mathcal{O})$-orbits together form a $12$-dimensional variety. The closure of this variety is therefore an irreducible component of maximal dimension and is the elusive third component of $\cyc$. As far as we know, this is the first recorded example of an irreducible component of maximal possible dimension in a cyclic convolution variety which is not simply the closure of a single $G(\mathcal{O})$-orbit.

\section{Acknowledgements} 
We thank Prakash Belkale for providing feedback on an earlier version of this manuscript and a referee for many helpful comments on the exposition.

 \end{document}